\documentclass{birkjour}
\usepackage{amssymb}
\newtheorem{theorem}{Theorem}[section]
\newtheorem{corollary}[theorem]{Corollary}
\newtheorem{lemma}[theorem]{Lemma}
\newtheorem{proposition}[theorem]{Proposition}
\theoremstyle{definition}
\newtheorem{definition}[theorem]{Definition}

\theoremstyle{remark}
\newtheorem{remark}[theorem]{Remark}
 \numberwithin{equation}{section}

\newcommand{\D}{\mathbb{D}}
\newcommand{\A}{\mathcal{A}}
\newcommand{\C}{\mathbb{C}}
\newcommand{\B}{\mathcal{B}}
\newcommand{\R}{\mathbb{R}}
\newcommand{\N}{\mathbb{N}}

\newcommand{\dpi}{\frac{1}{2\pi}\int_{0}^{2\pi}}
\newcommand{\cH}{\mathcal{H}}

\renewcommand{\Re}{\operatorname{Re}}
\renewcommand{\int}{\intop}
\renewcommand{\d}{\mathrm{d}}
\newcommand{\dtr}{\mathfrak {D}_{\mathcal R}}

\begin{document}

%
%
%
%
%
%
%
%
%

\title[Dirichlet-to-Robin Operators]{Dirichlet-to-Robin Operators via Composition Semigroups}

\author[L. Perlich]{Lars Perlich}

\address{%
Institute of Analysis\\
Technische Universit\"{a}t Dresden\\
Germany}

\email{lars.perlich@mailbox.tu-dresden.de}

\thanks{The author is supported by S\"achsisches Landesstipendium.}





\begin{abstract}
We show well-posedness for an evolution problem associated with the Dirichlet-to-Robin operator for certain Robin boundary data. Moreover, it turns out that the semigroup generated
by the Dirichlet-to-Robin operator is closely related to a weighted
semigroup of composition operators on an appropriate Banach space of
analytic functions.
\end{abstract}

\maketitle
\section{Introduction}

In recent years, the Dirichlet-to-Neumann operator has been studied
intensively. In the beginning of the 20th century, these operators
were dealt with theoretically, while in the 1980s and 1990s they were
used to analyze inverse problems to determine coefficients of a differential
operator. These problems apply, e.g., to image techniques in medicine
and also to find defects in materials. \\
According to Arendt and ter Elst, the Dirichlet-to-Neumann operator
can be obtained as an example of an operator associated with $m$-sectorial
forms, see \cite{ArEl12}. Using methods from function theory, our
purpose is to give an alternative approach to Poincar\'{e}-Steklov operators
and the related semigroups on boundary spaces of Banach spaces of
analytic functions. It turns out, as pointed out by Lax \cite{Lax02},
that there is a surprising connection between semigroups of composition
operators on spaces of harmonic functions on the unit disk referring
to a specific semiflow and the Dirichlet-to-Neumann operator. In fact,
we can extend this observation to the Laplace equation with Robin
boundary conditions on Jordan domains in $\C$.
More precisely, we study the evolution problem

\begin{equation}
\begin{cases}
\partial_{t}u-g\cdot u-G\cdot\partial_{z}u=0 &\textrm{on } (0,\infty)\times\partial\Omega,\\
-\Delta u=0 & \textrm{on } (0,\infty)\times\partial\Omega,\\
u(0,\cdot)=u_{0} & \textrm{on }\partial\Omega,
\end{cases}\label{eq:DTR}
\end{equation}
where $\Omega\subsetneq\C$ is a Jordan domain and $G$ and $g$ are boundary values of appropriate holomorphic functions on $\Omega$. We prove well-posedness of (\ref{eq:DTR}) in various spaces of distributions on $\partial\Omega$ including the scale of $L^p$-spaces.
As mentioned above, our approach does not use form methods but the theory of (weighted) composition operators on spaces of holomorphic and harmonic functions (for the moment only) on planar domains. Our method  appears to be restricted to problems involving the Laplace operator, while the variational approach  to Dirichlet-to-Neumann and Dirichlet-to-Robin operators using the theory of forms is quite flexible with respect the choice of elliptic operators in the domain $\Omega$. However, there it seems difficult to handle coefficients in front of the associated Neumann derivative (at least, we do not see how to handle them). Here, we can allow a large class of coefficient functions $G$ and $g$. In particular, it may happen that $G$ degenerates at one point on the boundary. Moreover, using our method, we  can define Dirichlet-to-Neumann and Dirichlet-to-Robin operators on several spaces of distributions.

This article is organized as follows. In Section 2 we introduce the notion
of admissible spaces  which is eventually our tool to solve the above posed evolution problem. We discuss some examples of admissible spaces, and we investigate
corresponding boundary spaces. Then, in Section 3, we examine
the connection between certain Poincar\'{e}-Stecklov operators, namely
Dirichlet-to-Neumann and Dirichlet-to-Robin operators, and weighted
semigroups of composition operators, and prove our main theorem. 

\section{Admissible spaces}

Initiated by the famous paper by Berkson and Porta \cite{BP78}, semigroups
of composition operators were studied intensively by many authors
on various spaces of holomorphic functions defined on the unit disk,
see, for example, \cite{AMW17,Bl13,Ko90,Sis86,Sis98}. In our approach,
we consider (weighted) semigroups of composition operators on spaces
of harmonic and holomorphic functions which are defined on a simply
connected domain $\Omega\subsetneq\C$ bounded by a Jordan curve.
To give the definition of such a semigroup, we need the notion of
a semiflow of holomorphic functions. 

\subsection*{Semiflows of holomorphic functions}
\begin{definition}
\label{def:flow}Let $\Omega\subsetneq\C$ be simply connected. Let
$\varphi:\Omega\rightarrow\Omega$ be holomorphic (we write $\varphi\in\cH(\Omega)$)
such that for every $t>0$ the fractional iterates $\varphi_{t}$
are holomorphic selfmaps in $\Omega$. A family $(\varphi_{t})_{t}$
is called a semiflow of holomorphic functions if it satisfies the
following properties: 
\begin{enumerate}
\item $\varphi_{0}(z)=z$ for all $z\in\Omega$, 
\item $\varphi_{s+t}(z)=\varphi_{s}(\varphi_{t}(z))$ for all $s,t>0$ and
$z\in\Omega$, 
\item $\varphi_{t}(z)\rightarrow z$ as $t\rightarrow0^{+}$ for all $z\in\Omega$. 
\end{enumerate}
\end{definition}
Given a semiflow $(\varphi_{t})_{t}$ we define its generator by 
\[
G(z)\colon=\lim_{t\to0^{+}}\frac{\varphi_{t}(z)-z}{t}
\]
 for every $z\in\Omega$. 

Since $\Omega$ is simply connected, by the Riemann mapping theorem
there exists a conformal map $k:\Omega\to\D$, and thus every semiflow
on $\Omega$ can be written in terms of a semiflow on the unit disk.
Let $(\varphi_{t})_{t}$ be a semiflow on $\D$. As a consequence
of the chain rule, the generator of $(\psi_{t})_{t}:=(k^{-1}\circ\varphi_{t}\circ k)_{t}$
can be written in terms of the generator of $(\varphi_{t})_{t}$.
For all holomorphic selfmaps $\varphi$ in the unit disk which are
not automorphisms, the embeddability into a semiflow can be characterized
in terms of the Denjoy-Wolff point of $\varphi$, see for instance
\cite{El02}. The Denjoy-Wolff point is defined as the unique fixed
point of a holomorphic selfmap in the unit disk which is not an automorphism
in the unit disk. Such a point can be found in the interior of the
unit disk as well as on the boundary. Thus we can use appropriate
M\"obius transforms to shift an interior Denjoy-Wolff point to zero
and a Denjoy-Wolff point on the boundary to 1. In our case, the representation
of $(\psi_{t})_{t}$ on $\Omega$ in terms of a semiflow on the unit
disk gives also the unique fixed point of every $\psi_{t}$ as
$k^{-1}(b)$ where $b$ is the Denjoy-Wolff point of $(\varphi_{t})_{t}$.
From the theory of differential equations, we obtain that $\varphi_{t}$
is univalent for every $t>0$, hence the same is true for $\psi_{t}$. \\
Let $b\in\bar{\D}$ be the Denjoy-Wolff point of a semiflow $(\varphi_{t})_{t}$
in $\mathcal{H}(\D)$. Then, by \cite{BP78}, the generator of $(\varphi_{t})_{t}$
is given by the Berkson and Porta formula 
\begin{equation}
G(z)=F(z)(\bar{b}z-1)(z-b),\label{eq:BP}
\end{equation}
where $F\colon\mathbb{D}\rightarrow\mathbb{C}$ is holomorphic and
$\Re(F(z))\geq0\,(z\in\D)$. It is also well known that $G$ is holomorphic
in $\D$ and that $\frac{d}{dt}\varphi_{t}=G(\varphi_{t})$.
In fact, if a holomorphic function $G\colon\D\rightarrow\C$ extends
continuously to $\bar{\D}$ and $\Re(G(z)\bar{z})\leq0$ for every
$z\in\D$, then $G$ is the generator of a semiflow in $\D$, see
\cite[Thm 1]{ARS96}. Conversely, a generator of a semiflow need not
extend continuously to the closure of $\D.$ On the other hand, note
that, by Fatou's theorem, a generator $G$ has radial limits almost
everywhere since the function $F$ is the composition of a bounded
holomorphic function and a M\"obius transform. The angle condition at
the boundary still holds.
\begin{lemma}
Let $(\varphi_{t})_{t}$ be a semiflow in the unit disk and $G$ its
generator. Then $\Re(G(z)\bar{z})\le0,\,\text{for a.e. }z\in\partial\D$.
\end{lemma}

\begin{proof}
Let $b\in\bar{\D}$ be the Denjoy-Wolff point of $(\varphi_{t})_{t}$.
Then, by \cite{BP78}, the generator is given by (\ref{eq:BP}) and
radial limits exist almost everywhere. For $z\in\partial\D$ we have

\begin{align*}
\Re(F(z)(\bar{b}z-1)(z-b)\bar{z})  &=  \Re(F(z)(\bar{b}-\bar{z})(z-b))\\
  &=  \Re(-F(z)|z-b|^{2})\\
  &\leq  0.
\end{align*}
\end{proof}
The same result holds true for generators of semiflows
on Jordan domains. 
\begin{lemma}
\label{lem:AngGenO}Let $\Omega\subsetneq\C$ be a Jordan domain.
Let $(\varphi_{t})_{t}$ be a semiflow in $\Omega$ and $G$ its generator.
Then $\Re(G(x)\overline{\nu(x)})\le0,\text{ for a.e.}\,x\in\partial\Omega$,
where $\nu(x)$ is the normal vector at $x.$
\end{lemma}

\begin{proof}
Let $k\colon\Omega\rightarrow\D$ be conformal. Therefore $(\psi_{t})_{t}=(k\circ\varphi_{t}\circ k^{-1})_{t}$
is a semiflow in the unit disk. Let $\tilde{G}$ be the generator
of $(\psi_{t})_{t}$. Then $\Re(\tilde{G}(z)\bar{z})\le0,\,\text{for a.e. }z\in\partial\D$.
\\
 For $z=k(x)\in\D$, we have 
\begin{align}
\tilde{G}(k(x))=\lim_{t\rightarrow0}\frac{d}{dt}\psi_{t}(k(x)) &=  \lim_{t\rightarrow0}\frac{d}{dt}k\circ\varphi_{t}(x)\nonumber \\
  &=  \lim_{t\rightarrow0}k'(\varphi_{t}(x))\frac{d}{dt}\varphi_{t}(x)\nonumber \\
  &=  k'(x)G(x).\label{eq:konfGen}
\end{align}
The function $k$ extends continuously to $\bar{\D}$ (see \cite[Thm. 2.6]{Po92})
and has non-vanishing angular derivative a.e. (see \cite[Thm. 6.8]{Po92}). Furthermore, for $x\in\partial\Omega$, we have $\nu(x)=\frac{k(x)}{k'(x)}$. For every $x\in\partial\Omega$ there exists a unique $z\in\partial\D$
such that $k(x)=z$, so 

\begin{align*}
\Re(G(x)\overline{\nu(x)})  &=  \Re\left(\frac{\tilde{G}(z)}{k'(x)}\overline{\left(\frac{z}{k'(x)}\right)}\right)\\
  &=  \frac{1}{|k'(x)|^{2}}\Re(\tilde{G}(z)\overline{z})\\
  &\leq  0.
\end{align*}
 
\end{proof}
Next, we transfer the characterization of generators of semiflows
in the unit disk given above to Jordan domains.

\begin{proposition}
\label{prop:CharGen}Let $G:\Omega\to\C$ be holomorphic, where $\Omega\subsetneq\C$
is simply connected.\\
(I) If $\partial\Omega$ is Dini-smooth and $G$ extends continuously
to $\bar{\Omega}$ and $\Re(G(x)\overline{\nu(x)}\le0$ for a.e. $x\in\partial\Omega$,
then $G$ is the generator of a semiflow in $\Omega$.\\
(II) If for every conformal map $k:\Omega\to\D$ there exists $\tau\in\bar{\Omega}$
and a holomorphic function $F:\D\to\C$ with positive real part such
that 
\begin{equation}
G(x)=\frac{F\circ k(x)(\overline{k(\tau)}k(x)-1)(k(x)-k(\tau))}{k'(x)}\quad(x\in\Omega),\label{eq:konfBP}
\end{equation}
 $G$ is the generator of a semiflow in $\Omega$. In this case, we
say that $G$ admits a conformal Berkson and Porta representation.
\end{proposition}

\begin{proof}
(I) Let $k:\Omega\to\D$ conformal. Define $\tilde{G}(z)=k'(k^{-1}(z))G(k^{-1}(z))$
for $z\in\D$. Then $\tilde{G}$ is a holomorphic function which admits
a uniformly continuous extension to $\bar{\D}$, by \cite[Thm 3.5]{Po92}.
Moreover, for $z\in\partial\D$,
\begin{align*}
\Re(\tilde{G}(z)\bar{z}) &= \Re\left(G(k^{-1}(z))\frac{k'(k^{-1}(z))}{k'(k^{-1}(z))}\overline{\left(\frac{k(k^{-1}(z))}{k'(k^{-1}(z))}\right)}\right)\\
 &= \Re(G(k^{-1}(z))\overline{\nu(k^{-1}(x)})\frac{1}{\left|k'(k^{-1}(z))\right|}\\
 &\le 0.
\end{align*}
So we can apply \cite[Thm. 1]{ARS96} which shows that $\tilde{G}$
is the generator of a semiflow $\psi_{t}$ in $\D$, and by (\ref{eq:konfGen})
$G$ is the generator of the semiflow $\left(\varphi_{t}\right)_{t}=\left(k^{-1}\circ\psi_{t}\circ k\right)_{t}$
.\\
(II) The function $(k'\cdot G)\circ k^{-1}:\D\to\C$ is given by the
Berkson and Porta formula, hence it is the generator of a semiflow
in $\D$ with Denjoy-Wolff point $b=k(\tau)$. The assertion follows
again by (\ref{eq:konfGen}).
\end{proof}

\subsection*{Weighted semigroups of composition operators}

Semiflows of holomorphic mappings lead to semigroups of composition
operators on spaces of holomorphic functions. Let
$\Omega\subset\C$ be simply connected, and consider the Frech\'{e}t space
$\cH(\Omega,\C)$ equipped with the topology of uniform convergence
on compact subset of $\Omega$. Let $(K_{n})_{n}$ be an increasing
sequence of compact subsets of $\Omega$ such that $\bigcup_{n}K_{n}=\Omega$.
We define a sequence of seminorms on $\cH(\Omega,\C)$ as follows
\[
p_{n}(f)\colon=\sup_{z\in K_{n}}|f(z)|\quad(f\in\cH(\Omega,\C)),
\]
and a metric induced by these seminorms by 
\[
d(f,g)\colon=\sum_{n=1}^{\infty}2^{-n}\frac{p_{n}(f-g)}{p_{n}(f-g)+1}\quad(f,g\in\cH(\Omega,\C)).
\]
For a given semiflow $(\varphi_{t})_{t}$, we define a family of composition
operators $(T_{t})_{t\ge0}$ acting on $\cH(\Omega,\C)$
as follows
\begin{align}
\label{eq:CO} T_{t}\colon\cH(\Omega,\C) &\to \cH(\Omega,\C)\\
f &\mapsto f\circ\varphi_{t}.\nonumber 
\end{align}
By the definiton of semiflows, this family is an operator semigroup
which is, in particular, strongly continuous since for all $n\in\N$,
we have
\[
\sup_{z\in K_{n}}|f(\varphi_{t}(z))-f(z)|\overset{t\to0^{+}}{\to}0.
\]
This defintion makes also sense when the space $h(\Omega,\C)$ of
harmonic functions is under consideration. Since, by the Cauchy-Riemann
equations, for every function $u\in h(\Omega,\C)$, we have $u\circ\varphi_{t}\in h(\Omega,\C)$.
\begin{definition}
\label{def:SGCO}Let $X\subset\mathbb{H}(\Omega,\C)$
be a Banach space and $(\varphi_{t})_{t}$ a semiflow of holomorphic
functions in $\cH(\Omega)$ generated by $G$. The space $X$ is called
$(G)$-admissible if the family of operators $(T_{t})_{t\ge0}$ defined
by (\ref{eq:CO}) satisfies the following two conditions:
\begin{itemize}
\item [(i)] $X$ is invariant under
$T_{t},$ i.e., $T_{t}X\subset X$ for all $t\ge0$.
\item [(ii)] $(T_{t})_{t\ge0}$
is strongly continuous on $X$.
\end{itemize}
\end{definition}

Given a semigroup of composition operators $(T_{t})_{t\ge0}$ on a
$(G)$-admissible Banach space $X$, the generator $\Gamma$ admits
a special form:
\[
\Gamma f=\lim_{t\to0^{+}}\frac{T_{t}f-f}{t}=G\cdot f'\quad(f\in\text{dom}\Gamma).
\]
Note that $G\cdot f'$ is a directional derivative. This is true for holomorphic functions and harmonic functions as well, but for convenience we write $\nabla f$ instead of $f'$ for harmonic functions to distinguish products of complex numbers from inner products. 
\subsubsection*{Examples}

Typical choices for the space $X$ are the Bergman spaces
\[
\mathcal{A}^{p}(\D)\colon=\mathcal{H}(\D,\C)\cap L^{p}(\D,dA)\,(p\geq1),
\]
where $dA$ denotes the normalized Lebesgue measure on $\D$, and
the Hardy spaces
\[
\mathcal{H}^{p}(\D)\colon=\left\{ f\colon\D\rightarrow\C\vert f\in\mathcal{H}(\D,\C);\,\sup_{0<r<1}\left(\dpi|f(re^{it})|^{p}\d t\right)^{\frac{1}{p}}<\infty\right\} \,(p\geq1).
\]
 The invariance is a consequence of Littlewood's subordination principle,
and the strong continuity follows from the density of the polynomials
and the dominated convergence theorem, see \cite{Sis98}, which is
also a comprehensive survey on semigroups of composition operators.

Indeed, this result carries over to Bergman and Hardy spaces on simply
connected domains. The Bergman spaces can be defined analogously to
the Bergman spaces for functions in the unit disk. For the Hardy space,
we can give at least two definitions for simply connected domains,
see \cite{Dur70}, either using harmonic majorants or via approximating
the boundary of $\Omega$ by rectifiable curves. Both definitions
are equivalent when analytic Jordan domains are considered. We use
the definition in terms of harmonic majorants.
\begin{definition}
Let $\Omega\subsetneq\C$ be simply connected. For $p\in[1,\infty)$,
the Hardy space $\cH^{p}(\Omega)$ consists of those functions $f\in\cH(\Omega,\C)$
such that the subharmonic functions $|f|^{p}$ is dominated by a harmonic
function $u:\Omega\to\R$.
\end{definition}

Equipped with the norm $\left\Vert f\right\Vert _{\cH^{p}(\Omega)}:=(u_{0}(z_{0}))^{\frac{1}{p}}\,(f\in\cH(\Omega))$
where $z_{0}\in\Omega$ is some fixed point and $u_{0}$ is the least
harmonic majorant for $f$, the Hardy space over $\Omega$ is a Banach
space. As in the unit disk, functions in $\cH^{p}(\Omega)$ admit
non-tangential limits a.e. on $\partial\Omega$ and the boundary function
is in $L^{p}(\partial\Omega)$. For more details about Hardy spaces
over general domains, we refer to \cite[Ch. 10]{Dur70}.

\begin{proposition}
\label{prop:SGHardy}Let $\Omega\subsetneq\C$ be simply connected.
Let $(\varphi_{t})_{t}$ be a semiflow of holomorphic functions in
$\cH(\Omega)$ generated by $G$. The Hardy space $\cH^{p}(\Omega)$
($p\in[1,\infty)$) is $(G)$-admissible.
\end{proposition}

\begin{proof}
Let $k\colon\Omega\rightarrow\D$ be conformal.
Then there exists a semiflow $(\psi_{t})_{t}$ in $\cH(\D)$ such
that $(\varphi_{t})_{t}=(k^{-1}\circ\psi_{t}\circ k)_{t}$. By \cite[Cor. to Thm. 10.1]{Dur70},
$f\in\cH^{p}(\Omega)$ if and only if$f\circ k^{-1}\in\mathcal{H}^{p}(\D$).
This and Littlewood's subordination principle gives invariance since
\[
|f\circ\varphi_{t}|^{p}=|\underset{\in\cH^{p}(\Omega)}{\underbrace{\underset{\in\cH^{p}(\D)}{\underbrace{\underset{\in\cH^{p}(\D)}{\underbrace{f\circ k^{-1}}\circ\psi_{t}}}}\circ k}|^{p}.}
\]
Without loss of generality, we assume that $k^{-1}(0)=z_{0}$. Then,
by \cite[p. 168]{Du04}, we have
\[
\left\Vert f\circ\varphi_{t}-f\right\Vert _{\cH^{p}(\Omega)}=\left\Vert f\circ\varphi_{t}\circ k^{-1}-f\circ k^{-1}\right\Vert _{\cH^{p}(\D)}\overset{t\to0^{+}}{\to}0.
\]

\end{proof}

\begin{remark}
If we were using the definition of Hardy spaces by approximating level
curves (sometimes called Hardy-Smirnov spaces), the last proof would
involve boundary values of conformal maps. This would have forced
us to prescribe conditions concerning the boundary of $\Omega$. Therefore
it seems more appropriate to define Hardy spaces via harmonic majorants.
\end{remark}
\begin{proposition}
\label{prop:SGBergman}Let $\Omega\subsetneq\C$ be a Jordan domain.
Let $(\varphi_{t})_{t}$ be a semiflow of holomorphic functions in
$\cH(\Omega)$ generated by $G$. The Bergman space
$\A^{p}(\Omega)$ ($p\in[1,\infty)$) is $(G)$-admissible.
\end{proposition}

\begin{proof}
Let $k\colon\Omega\rightarrow\D$ be conformal.
Then there exists a semiflow $(\psi_{t})_{t}$ in $\cH(\D)$ such
that $(\varphi_{t})_{t}=(k^{-1}\circ\varphi_{t}\circ k)_{t}$. Thus,
for $f\in\A^{p}(\Omega)$, 
\begin{align*}
\int_{\Omega}|f\circ\varphi_{t}|^{p} &= \int_{\D}|f\circ\varphi_{t}\circ k^{-1}|^{p}|\frac{1}{k'}|^{2}\\
 &\le C\int_{\D}|f\circ k^{-1}\circ\psi_{t}|^{p}.
\end{align*}
The derivative of $k$ is non-vanishing in $\bar{\Omega}$, see \cite[Thm. 6.8]{Po92}.

For invariance, we only need to show that $f\circ k^{-1}\in\A^{p}(\D)$.
Indeed, 
\[
\int_{\D}|f\circ k^{-1}|^{p}\d A=\int_{\Omega}|f|^{p}|k'|^{2}\d A\le C\|f\|_{\A^{p}(\Omega)}^{p}<\infty.
\]
 Now Littlewood's subordination principle yields invariance.

By the same calculation, we obtain strong continuity
of $(T_{t})_t$ on $\A^{p}(\Omega)$ from strong continuity on $\A^{p}(\D)$.
\end{proof}
Further examples of holomorphic function spaces
on the unit disk which appear in the literature concerning semigroups
of composition operators are the Bloch space $\mathcal{B}$ and the
space BMOA as well as their subspaces $\mathcal{B}_{0}$ and VMOA.
On these spaces the question of strong continuity is much more delicate,
and in fact there is no nontrivial strongly continuous semigroup on
$\mathcal{B}$ and BMOA. So in these cases, one is studying so-called
maximal subspaces of strong continuity denoted by $\left[\varphi_{t},\mathcal{B}\right]$
and $\left[\varphi_{t},\text{BMOA}\right]$ such that a given semiflow
$(\varphi_{t})_{t}$ defines a strongly continuous semigroup of composition
operators on $\left[\varphi_{t},\mathcal{B}\right]$ resp. $\left[\varphi_{t},\text{BMOA}\right]$.
In \cite{Bl13} it has been shown that $\B_{0}\subseteq[\varphi_{t},\B]\subsetneq\B$
, and in the recent paper \cite{AMW17} the analogous result for BMOA
has been obtained, that is, $\mathrm{VMOA\subseteq}\left[\varphi_{t},\mathrm{BMOA}\right]\subsetneq\mathrm{BMOA}.$ 

It is also natural to consider weighted semigroups of composition
operators. Let $\Omega\subsetneq\C$ be simply connected. Let $\omega:\Omega\to\C$
be holomorphic. For $t\in\R_{+}$ we define a weight as follows
\begin{equation}
m_{t}=\frac{\omega(\varphi_{t})}{\omega}.\label{eq:weight}
\end{equation}
For a family of composition operators $(T_{t})_{t\ge0}$ on $\cH(\Omega,\C)$
with semiflow $\varphi_{t}\in\cH(\Omega)$, we define a family of
weighted composition operators as follows

\begin{align}
S_{t}\colon\cH(\Omega,\C) \to \cH(\Omega,\C)\nonumber \\
f &\mapsto m_{t}\cdot T_{t}f.\label{eq:WSG}
\end{align}
This is again an operator semigroup on $\cH(\Omega,\C)$ and also
on $h(\Omega,\C)$ but the question of strong continuity is more difficult
since it depends heavily on the choice of $\omega$. 

Special weights we are interested in are so-called cocycles.
\begin{definition}
Let $(\varphi_{t})_{t}$ be a semiflow in $\mathcal{H}(\Omega,\Omega)$.
A family of holomorphic functions $m_{t}\colon\Omega\rightarrow\C,t\geq0,$
is called cocycle if
\begin{enumerate}
\item $m_{0}(z)=1,\,z\in\Omega$, 
\item $m_{s+t}(z)=(m_{s}\cdot m_{t})(\varphi_{t}(z))$ for all $t,s\geq0$
and $z\in\Omega$, 
\item $t\mapsto m_{t}(z)$ is continuous for every $z\in\Omega$. 
\end{enumerate}
\end{definition}

If there exists a holomorphic function $w\colon\Omega\rightarrow\C$
such that $m_{t}(z)=\frac{w(\varphi_{t}(z))}{w(z)},\,z\in\Omega,$
then the family $(m_{t})_{t}$ is called a coboundary of $(\varphi_{t})_{t}$.

It is easy to see that a family of cocycle weighted composition operators
is also an operator semigroup on $\mathbb{H}(\Omega,\C)$. Moreover,
given an arbitrary holomorphic function $g:\Omega\to\C$, we can easily
construct a cocycle to a semiflow $(\varphi_{t})_{t}$: for $t\ge0$, 

\begin{equation}
m_{t}(z)=\exp\left(\int_{0}^{t}g(\varphi_{s}(z))ds\right)\quad(z\in\Omega)\label{eq:cocy}
\end{equation}
is a cocycle. 
\begin{definition}
\label{def:admsibble}Let $(S_{t})_{t\ge0}$ be a weighted semigroup
of composition operators on $\mathbb{H}(\Omega,\C)$, cf. (\ref{eq:WSG}),
with semiflow generated by the holomorphic function $G:\Omega\to\C$
and cocycle weight in terms of a holomorphic function $g:\Omega\to\C$,
see (\ref{eq:cocy}) . A Banach space $X\subset\mathbb{H}(\Omega,\C)$
is called $(g,G)$-admissible if it satisfies the following two conditions:
\begin{enumerate}
\item [(i)] $X$ is invariant under
$S_{t},$ i.e., $S_{t}X\subset X$ for all $t\ge0$.
\item [(ii)] $(S_{t})_{t}$ is strongly
continuous on $X$.
\end{enumerate}
\end{definition}
Let $X\subset\cH(\Omega,\C)$ be $(g,G)$-admissible. Then the generator $\Gamma$ of $(S_{t})_{t\ge0}$ is given by 
\[
\Gamma f=g\cdot f+G\cdot f'\quad(f\in\text{dom}\Gamma).
\]

\subsubsection*{Examples}

In \cite[Theorem 2]{Ko90} it has been shown that for certain holomorphic
functions $g:\Omega\to\C$ and their associated cocycles $m_{t}$
as in (\ref{eq:cocy}), and a semiflow $(\varphi_{t})_{t}$ generated by
$G:\Omega\to\C$,the Hardy space $\cH^{p}(\D)\,(p\in[1,\infty))$
is $(g,G$-admissible in the sense of Definition \ref{def:admsibble}. By a slight adjustment of
the arguments in Proposition \ref{prop:SGHardy}, we obtain the result
for Hardy spaces over simply connected sets.
\begin{lemma}
Let $\Omega\subsetneq\C$ be simply connected. Let $g\colon\Omega\rightarrow\C$
be a holomorphic function such that $\underset{z\in\Omega}{\sup}\Re g(z)<\infty$,
and let $(\varphi_{t})_{t}$ be a semiflow in $\mathcal{H}(\Omega)$
with generator $G$. Then $\mathcal{\cH}^{p}(\Omega)\,(p\in[1,\infty))$
is $(g,G)-$admissible. 
\end{lemma}

\begin{proof}
Invariance follows by boundedness of $m_{t}$ and Proposition \ref{prop:SGHardy}.
To show strong continuity, we use the same technique as in Proposition
\ref{prop:SGHardy}, too. Since the real part of $g\circ k^{-1}$ is
bounded as well, we obtain the assertion from \cite[Theorem 1]{Sis86}.
\end{proof}
Indeed, the proof of \cite[Theorem 1]{Sis86} works as well for a
family of $m_{t}$-weighted composition operators on the Bergman space
$\mathcal{A}^{p}(\Omega)$, where $\Omega$ is a Jordan domain. 

\begin{lemma}
\label{lem:scon}Let $\Omega\subsetneq\C$ be a Jordan domain. Let
$g\colon\Omega\rightarrow\C$ be a holomorphic function such that
$\underset{z\in\Omega}{\sup}\Re g(z)<\infty$, and let $(\varphi_{t})_{t}$
be a semiflow in $\mathcal{H}(\Omega)$ with generator $G$. Then
\textup{$\mathcal{A}^{p}(\Omega)\,(p\in[1,\infty))$ is $(g,G)$-admissible. }
\end{lemma}

\begin{proof}
It suffices to prove the statement for $\A^{p}(\D)=\A^{p}$ and then
apply the same technique as in Proposition \ref{prop:SGBergman}.
Due to Siskakis \cite[Theorem 1]{Sis86}, strong continuity for a
weighted SGCO on $\mathcal{H}^{p}(\D)$ is achieved if $\limsup_{t\to0}\|m_{t}\|_{\infty}\leq1$
which is satisfied by our assumptions on $g$, see \cite[Lemma 3.1]{Ko90}.
To prove the assertion, we can simply follow the steps in the proof
of \cite[Theorem 1]{Sis86}. \\
For all $t\geq0$ we have $m_{t}\in\mathcal{H}^{\infty}(\D)$. This
and the cocycle properties yield that $(S_{t})_{t}$ defines a family
of bounded operators on $\mathcal{A}^{p}$. Let $f\in\mathcal{A}^{p}$.
By Littlewood's subordination principle we get that 
\begin{equation}
\|S_{t}f\|_{\A^{p}}^{p}\leq\|m_{t}\|_{\infty}\left(\frac{1+|\varphi_{t}(0)|}{1-|\varphi_{t}(0)|}\right)^{p}\|f\|_{\A^{p}}^{p},\label{eq:LSP}
\end{equation}
thus $\|S_{t}\|_{\mathcal{L}(\mathcal{A}^{p},\mathcal{A}^{p})}<\infty$
for all $t\ge0$.\\
First, we prove strong continuity if $p>1$. Let $(t_{n})_{n\in\mathbb{N}}$
be a sequence such that $t_{n}\overset{n\to\infty}{\to}0$. Then we
have $\limsup_{n\to\infty}\|S_{t_{n}}f\|_{2}\leq\|f\|_{2}$. Since
$\mathcal{A}^{p}$ is reflexive and by (\ref{eq:LSP}), after passing
to a subsequence again denoted by $(t_{n})_{n\text{\ensuremath{\in\mathbb{N}}}}$,
the sequence $(S_{t_{n}}f)_{t_{n}}$ is weakly convergent. The weak
limit is $f$ because $(S_{t_{n}}f(z))_{t_{n}}\to f$ for all $z\in\D$.
By lower-semicontinuity of the $\mathcal{A}^{p}$ norm, $\|f\|_{\mathcal{A}^{2}}\leq\liminf_{n\rightarrow\infty}\|S_{t_{n}}f\|_{\mathcal{A}^{p}}$,
and thus $\|S_{t_{n}}f\|_{\mathcal{A}^{p}}\to\|f\|_{\mathcal{A}^{p}}$. This yields the desired strong continuity.\\
To show strong continuity in the case $p=1,$ we
use that $\A^{q}$ ($q>1)$ is dense in $\A^{1}$. Let $\epsilon>0$.
For every $f\in\A^{1}$ there exists $g\in\A^{q}$ such that $\left\Vert f-g\right\Vert _{\A^{1}}<\frac{\epsilon}{(\left\Vert S_{t}\right\Vert +1)2}$. Moreover,
\begin{align*}
\left\Vert S_{t}f-f\right\Vert _{\A^{1}} &\le \left\Vert S_{t}f-S_{t}g\right\Vert _{A^{1}}+\left\Vert S_{t}g-g\right\Vert _{A^{1}}+\left\Vert f-g\right\Vert _{A^{1}}\\
 &\le \left\Vert S_{t}f-S_{t}g\right\Vert _{A^{1}}+\left\Vert S_{t}g-g\right\Vert _{A^{q}}+\left\Vert f-g\right\Vert _{A^{1}}\\
 &\le (\left\Vert S_{t}\right\Vert +1)\left\Vert f-g\right\Vert _{A^{1}}+\left\Vert S_{t}g-g\right\Vert _{A^{q}}.
\end{align*}
Since $q>1$, for all $\epsilon>0$ there exists a sufficiently small
$t>0$ such that $\left\Vert S_{t}g-g\right\Vert _{A^{q}}<\frac{\epsilon}{2}.$
Thus $\left\Vert S_{t}f-f\right\Vert _{\A^{1}}\to0$ as $t\to0^{+}$.\\
\end{proof}
\begin{remark}
Several authors are especially interested in semigroups of composition
operators weighted by the derivative of the semiflow $(\varphi_{t})_{t}$
with respect to the complex varibale, i.e., 
\[
S_{t}f:=\varphi'_{t}\cdot f\circ\varphi_{t}\quad(f\in X).
\]
See for example the recent paper \cite{ArOl17}.

Indeed, this weight is a cocycle given by 
\[
m_{t}(z)\colon=\varphi'_{t}(z)=\exp\left(\int_{0}^{t}G^{'}(\varphi_{s}(z))\d s\right).
\]
\end{remark}

\subsection*{Boundary spaces}

Finding boundary values of holomorphic functions is a fundamental
problem in function theory. Strong results concerning the boundary
values of functions in Hardy spaces are Fatou's theorem and the theorem
by F. and M. Riesz. But, in many spaces of holomorphic functions,
convergence to boundary values in a nontangential sense is a rather
strong condition. Therefore we consider boundary values in a weaker
sense, namly in the sense of distributions.

Let $\Omega\subsetneq\C$ be a Jordan domain. This restriction guarantees
existence and nonvanishing of boundary values of derivatives of conformal
maps defined on $\Omega$. Up to now, we are not sure if the established
theory works for rectifiable boundaries as well. 

In what follows, we are exploring boundary distributions of functions
in Banach spaces $X\subset\mathbb{H}(\Omega,\C)$. Our first aim is
to define the boundary space of $X$ consisting of appropriately defined
distributional boundary values of elements of $X$. 
\begin{definition}
\label{def:BS}Let $\Omega\subsetneq\C$ be a Jordan domain. Let $X\subset\mathbb{H}(\Omega,\C)$
be a Banach space. If for every $f\in X$ there exists a uniquely
defined boundary distribution $f^{*}\colon\partial\Omega\to\C$ in
the following sense 
\[
\lim_{r\to1^{-}}\int_{\partial\Omega}f_{r}\cdot\phi(x)\d x=\left\langle f^{*},\phi\right\rangle 
\]
for every $\phi\in C^{\infty}(\partial\Omega)$, where $f_{r}(z):=f(k^{-1}(rk(z))$,
and $k:\Omega\to\D$ is any conformal map, then we denote the set
consisting of all such boundary values by $\partial X$. If there
exists an isomorphism $\text{Tr}:X\to\partial X$, then $\partial X$
is called the boundary space corresponding to $X$.
Moreover, we define a norm on $\partial X$ by $\|f^{*}\|_{\partial X}=\|f\|_{X}$
for every $f^{*}\in\partial X$.
\end{definition}

\subsubsection*{Examples. }

A first (though artificial) example is the space $X=\mathcal{A}$
where $\mathcal{A}$ denotes the disk algebra. The restriction to
the boundary is an isometric homomorphism from $\mathcal{A}$ into
$C(\partial\D)$. So $\mathcal{A}$ is a Banach subalgebra of $C(\partial\D)$
which is even maximal due to Wermer's maximality theorem. Thus the
boundary space $\partial X$ can be defined as the space of continuous
functions on $\partial\D$ which are holomorphically extendable to
$\D$. \\
Let $p\in[1,\infty)$ and define $X$ as the Hardy space $\cH^{p}(\D)$.
Then it is well known that every function in $\cH^{p}(\D)$ has nontangential
limits a.e. and the boundary function is in $L^{p}(\partial\D)$.
For a comprehensive overview, we refer especially to \cite[Chapter 3]{Dur70}.
These boundary functions form a closed subspace of $L^{p}(\partial\D$)
which consists of those function in $L^{p}(\partial\D)$ with vanishing
negative Fourier coefficients. Note that this theory is almost applicable
when the analogously defined Hardy space $h^{p}$ of harmonic functions
is considered. However, the case $p=1$ appears to be different. The
boundary space on $h^{1}$ consists of finite Borel measures on the
unit circle. 

In both examples, the boundary space inherits some properties of the
underlying space of holomorphic functions. Moreover, by the Luzin-Privalov
theorem, a holomorphic function is in either case identically zero
if the boundary function vanishes on a set of positive measure. Given
a function in one of the two boundary spaces from the examples above,
we can recover the holomorphic function in $X$ via Cauchy's integral
formula and the Poisson integral as well which acts as an isometric
isomorphism between $X$ and $\partial X$.

\subsubsection*{Boundary distributions of Bergman functions.}

The theory of boundary values for functions in Hardy spaces on the
unit disc is well established. The question of boundary functions
is much more complicated if one wishes to work on Bergman spaces.
In fact, the Bergman spaces contain functions which do not admit nontangential
or radial limits almost everywhere, such as the Lacunary series. So
it seems more appropriate to define boundary values in the sense of
distributions. To establish such distributional boundary values, we
emphasize a connection between Hardy and Bergman spaces. For simplicity
we use the notation $\mathcal{A}^{p}:=\mathcal{A}^{p}(\D)$ and $\mathcal{H}^{p}:=\mathcal{H}^{p}(\D),p\geq1$.
The following theorem can be found in \cite[Lem. 4]{Du04}. 
\begin{theorem}
\label{thm:.p=00003D00003D1} \textup{If $f\in\mathcal{A}^{1}$ and
$F$ is an antiderivative of $f$, then $F\in\mathcal{H}^{1}$.}
\end{theorem}

For $p\in[1,\infty)$, Theorem \ref{thm:.p=00003D00003D1} can be
generalized to $f\in\mathcal{A}^{p}$ in the following way. 
\begin{theorem}
\label{thm:anti}Let $f\in\mathcal{A}^{p}\,(p\geq1)$ and $F$ an
antiderivative of $f$. Then $F\in\mathcal{H}^{p}$.
\end{theorem}

\begin{proof}
Let $\varepsilon\in(0,1)$. Then we have 
\begin{align*}
F(z)  &=  \int_{\varepsilon z}^{z}f(w)\d w-F(\varepsilon z)\overset{(w=tz)}{=}\int_{\varepsilon}^{1}f(tz)z\d t-F(\varepsilon z)
\end{align*}
To estimate $M_{p}(r,F)$, we examine the following two integrals
\begin{align*}
M_{p}(r,F) &=  \left(\dpi|F(re^{it})|^{p}\d t\right)^{\frac{1}{p}}\\
  &=  \left(\dpi\left|\int_{\varepsilon}^{1}f(sre^{it})re^{it}ds+F(\varepsilon re^{it})\right|^{p}\d t\right)^{\frac{1}{p}}\\
 &\le  \underset{=\colon I_{1}}{\underbrace{\left(\dpi\left|\int_{\varepsilon}^{1}f(sre^{it})re^{it}\d s\right|^{p}\d t\right)^{\frac{1}{p}}}}+\underset{=\colon I_{2}}{\underbrace{M_{p}(\varepsilon r,F)}}.
\end{align*}
 For the first term we have
\begin{align*}
I_{1}^{p}  &= \dpi\left|\int_{\varepsilon}^{1}f(sre^{it})re^{it}\d s\right|^{p}\d t\\
 &\leq r^{p}\dpi\left(\int_{\varepsilon}^{1}|f(sre^{it})|\d s\right)^{p}\d t\\
&\leq  r^{p}\dpi\left(\int_{\varepsilon}^{1}|f(sre^{it})|^{p}\d s\right)\d t\\
 &\leq r^{p+1}\dpi\left(\int_{r\varepsilon}^{r}|f(ue^{it})|^{p}\frac{u}{u}\d u\right)\d t\\
 &\leq  \frac{r^{p}}{\varepsilon}\dpi\left(\int_{0}^{1}|f(ue^{it})|^{p}u\d u\right)\d t\\
 &\leq  \frac{r^{p}}{\varepsilon}\|f\|_{\mathcal{A}^{p}}^{p}.
\end{align*}
Without loss of generality, we assume $f(0)=0$. Thus we obtain for the second integral 
\begin{align*}
I_{2}^{p}  &=  \dpi|F(\varepsilon re^{it})|^{p}\d t\\
  &=  \dpi\left|\int_{0}^{1}f(s\varepsilon re^{it})\varepsilon re^{it}\d s\right|^{p}\d t\\
 &\leq (\varepsilon r)^{p}\dpi\left(\int_{0}^{1}|f(s\varepsilon re^{it})|^{p}\d s\right)\d t\\
  &=  (\varepsilon r)^{p}\dpi\left(\int_{0}^{\varepsilon r}|f(ue^{it})|^{p}\d u\right)\d t\\
  &\le  (\varepsilon r)^{p}\|f\|_{\mathcal{A}^{p}}^{p}.
\end{align*}
 Combining these results, we have
\[
M_{p}(r,F)\leq\left(\frac{r^{p}}{\varepsilon}+(\varepsilon r)^{p}\right)\|f\|_{\mathcal{A}^{p}}.
\]
Letting $r\rightarrow1^{-}$, the right-hand side is still finite
since $\varepsilon$ can be chosen arbitrarily in $(0,1)$.
\end{proof}
This theorem remains true if we replace $\D$ by a Jordan domain $\Omega\subsetneq\C$.
\begin{corollary}
\label{cor:D-O}Theorem \ref{thm:anti} remains true if $\D$ is replaced
by a Jordan domain $\Omega\subsetneq\C$.
\end{corollary}

\begin{proof}
By \cite[Cor. to Thm. 10.1]{Dur70}, it is enough
to show that $F\circ k^{-1}\in\mathcal{H}^{p}(\D$) for some conformal
mapping $k\colon\Omega\rightarrow\D$. Therefore, one can mostly copy
the proof of Theorem \ref{thm:anti}, noting that for $f\in\mathcal{A}^{p}(\Omega)$
one has $f\circ k^{-1}\cdot\frac{1}{k'}\in\mathcal{A}^{p}(\D)$. 

The derivative of $k$ does not vanish in $\bar{\Omega}$,
and so we have
\[
\int_{\D}|f\circ k^{-1}|^{p}\left|\frac{1}{k'}\right|^{\frac{1}{p}}\d A\le C\|f\circ k^{-1}\|_{\A^{p}(\D)}^{p}.
\]
It remains to show that $f\circ k^{-1}\in\A^{p}(\D)$: 
\[
\int_{\D}|f\circ k^{-1}|^{p}\d A=\int_{\Omega}|f|^{p}|k'|^{2}\d A\le C\|f\|_{\A^{p}(\Omega)}^{p}<\infty.
\]

\end{proof}
Now we can define distributional boundary values for Bergman functions. 
\begin{theorem}
\label{thm:DstrBV}Let $p\in[1,\infty)$. Every function $f\in\mathcal{A}^{p}(\D)$
admits a distributional boundary value in $W^{-1,q}(\partial\D):=(W^{1,p}(\partial\D))'$,
the dual space of $W^{1,p}(\partial\D$), where $q$ is the usual
conjugate exponent of $p.$%

\end{theorem}

\begin{proof}
Let $\varphi\in W^{1,p}(\partial\D)$. We denote by $F$ the antiderivative
of $f$, so we obtain 
\begin{align*}
\dpi f(re^{it})\varphi(e^{it})\d t  &=  \underset{=0}{\underbrace{\frac{1}{2\pi}F(re^{it})\varphi(e^{it})|_{0}^{2\pi}}}-\dpi F(re^{it})\varphi'(e^{it})\d t\\
  &\overset{r\rightarrow1^{-}}{\rightarrow}  -\left\langle T_{f},\varphi\right\rangle .
\end{align*}
This limit exists by using Theorem \ref{thm:anti}, H\"older's inequality,
and the dominated convergence theorem.
\end{proof}
\begin{corollary}
\label{thm:dbvp}Let $\Omega\subset\C$ be Jordan domain. Then every
function $f\in\mathcal{A}^{p}(\Omega)\,(p\in[1,\infty))$ admits a
distributional boundary value in $W^{-1,q}(\partial\Omega)$.
\end{corollary}

\begin{proof}
Let $\varphi\in W^{1,p}(\partial\Omega)$, and let
$k:\Omega\to\D$ be conformal. For $r\in(0,1)$ we define as usual
$f_{r}:\bar{\Omega}\to\C,x\mapsto f(k^{-1}(rk(x))).$ Thus $f_{r}\to f$
as $r\to1^{-}$. Then 
\begin{align}
\int_{\partial\Omega}f_{r}(x)\varphi(x)\d x &= \int_{\partial\D}f_{r}(k^{-1}(x))\varphi(k^{-1}(x))\left|\frac{1}{k'(x)}\right|^{2}\d x\nonumber \\
 &= \int_{0}^{2\pi}f_{r}(k^{-1}(e^{it}))\varphi(k^{-1}(e^{it}))\left|\frac{1}{k'(x)}\right|^{2}ie^{it}\d t.\label{eq:dbvo}
\end{align}
It is easy to show that $f_{r}\circ k^{-1}\in\A^{p}(\D)$ and $\varphi\circ k^{-1}\in L^{p}(\partial\D)$.
Since $k$ is a conformal, we also have $\varphi\circ k^{-1}\in W^{1,p}(\partial\D)$.
So, by \cite[Thm 6.8]{Po92}, we obtain convergence of the integral
(\ref{eq:dbvo}) as $r\to1^{-}$.
\end{proof}
Distributional boundary values of harmonic and holomorphic functions
defined on a simply connected domain with smooth boundary have been
studied in \cite{St84}. There it has been shown that a holomorphic
function admits a distributional boundary value if and only if it
lies in the Sobolev space $H^{-k}(\Omega):=W^{-k,2}(\Omega)$ for
some $k\in\mathbb{N}$, see \cite[Thm. 1.3]{St84}. Moreover, by \cite[Cor. 1.7]{St84},
for all $k\in\mathbb{N}$ the map $P$ defined by 
\begin{align*}
P:W^{-k-\frac{1}{2},2}(\partial\Omega)  &\to  H^{-k}(\Omega)\cap\cH(\Omega,\C)\\
T_{f}  &\mapsto  \langle P_{z},T_{f}\rangle=f(z),
\end{align*}
where $P_{z}$ is the Poisson kernel for $\Omega$, is an isomorphism.
The inverse is given by assigning the distributional boundary value
to a given function. Thus, functions in $H^{-k}(\Omega)\cap\cH(\Omega,\C)$
are uniquely determined by their boundary distributions. Therefore,
restricting the map $P$ to the boundary space $\partial\mathcal{A}^{p}(\Omega)$
for some $p\in[1,\infty)$, we can recover each function in $\mathcal{A}^{p}(\Omega)$
using the Poisson operator. 

\section{Dirichlet-to-Robin via composition semigroups}

In this section we work out our main result, the connection between
partial differential equations on the boundary associated with Poincar\'{e}-Steklov
operators and semigroups of composition operators on Banach spaces
of holomorphic functions.

\subsection*{The Lax semigroup}

Let $h\colon\partial\D\rightarrow\C$ be a 'nice' function and consider
the following elliptic equation

\begin{equation}
\begin{cases}
-\Delta u=0 & \textrm{in }\D,\\
u=h & \textrm{on }\partial\D.
\end{cases}\label{eq:DP.D}
\end{equation}
The Dirichlet-to-Neumann operator $\mathfrak{D}_{\mathcal{N}}$ maps
the function $h$ to the Neumann derivative of the solution of (\ref{eq:DP.D})
provided that a solution exists and is sufficiently regular. As it
is shown by Lax \cite{Lax02}, if $g\in C(\partial\D)$ or in $L^{2}(\partial\D)$,
the Dirichlet-to-Neumann operator generates the following semigroup
\begin{equation}
T_{t}h(z)\colon=u(ze^{-t})\quad(z\in\partial\D).\label{eq:lax}
\end{equation}
This semigroup solves the first order evolution equation associated
with the Dirichlet-to-Neumann operator
\begin{equation}
\begin{cases}
\partial_{t}u+\partial_{\nu}u=\ 0 & \textrm{on }(0,\infty)\times\partial\D,\\
-\Delta u=0 & \textrm{on }(0,\infty)\times\D\\
u(0,\cdot)=h& \textrm{on }\partial\D.
\end{cases}\label{eq:DtN}
\end{equation}

In fact, the semigroup (\ref{eq:lax}) is an unweighted semigroup
of composition operators on $h^{p}(\D)$ if $h\in\partial h^{p}(\D)\subseteq L^{p}(\partial\D)$
($p\in(1,\infty)$) with associated semiflow $(\varphi_{t})_t$ given by $\varphi_t(z)=ze^{-t}\,(z\in\D)$.
The generator  is given by $\ G(z)=-z=-\nu(z)\,\,(z\in\D)$,
and therefore the generator of the semigroup (\ref{eq:lax}) is $\Gamma u=-\nu \cdot \nabla u\,\,(u\in\text{dom}(\Gamma)\subset h^{p}(\D))$.
So, for $h\in\text{dom}(\mathfrak{D}_{\mathcal{N}})\subset L^{p}(\partial\D)$
and $u\in h^{p}(\Omega)$ the solution to (\ref{eq:DP.D}), 
\begin{align*}
-\mathfrak{D}_{\mathcal{N}}h &= -\partial_{\nu}u\\
 &= -\nu \cdot \nabla u\\
 &=\operatorname{Tr}( \Gamma u).
\end{align*}

\subsection*{Dirichlet-to-Neumann on $\Omega$}

Replacing $\D$ by a simply connected domain $\Omega$ with Dini-smooth
boundary in (\ref{eq:DP.D}) and (\ref{eq:DtN}), we obtain a similar
correspondence. Let $k:\Omega\to\D$ be conformal, then $\nu(z)=\frac{k(z)}{k'(z)}|k'(z)|$
is the unit normal vector at $z\in\partial\Omega$. Since $\partial\Omega$
is Dini-smooth, $k\in C^{1}(\bar{\Omega})$ by \cite[Thm. 3.5]{Po92}.
Thus $G(z)=-\frac{k(z)}{k(z)'}\,\,(z\in\Omega)$ is holomorphic in
$\Omega$ and uniformly continuous on $\bar{\Omega}$, and moreover,
$\Re(-G\bar{\nu})\le0$ on $\partial\Omega$. So, by Proposition \ref{prop:CharGen}
(I), $G$ generates a semiflow in $\cH(\Omega)$. Therefore, we obtain
the following relation between the Dirichlet-to-Neumann operator on
$\partial h^{p}(\Omega)\subset L^{p}(\partial\Omega)$ and the unweighted
semigroup of composition operators on $h^{p}(\Omega)$. Let $u\in h^{p}(\Omega)$
be the solution to
\[
\begin{cases}
-\Delta u=0 & \textrm{in }\Omega,\\
u=h & \textrm{on }\partial\Omega,
\end{cases}
\]
where $h\in L^{p}(\partial\Omega)$. Then, for $h\in\text{dom}(\mathfrak{D}_{\mathcal{N}})$,
\begin{align*}
-\mathfrak{D}_{\mathcal{N}}h &= -\partial_{\nu}u\\
 &= \operatorname{Tr}((G\cdot \nabla u)|k'|),
\end{align*}
and $\Gamma u:=G\cdot \nabla u$ is the generator of an unweighted semigroup
of composition operators on $h^{p}(\Omega)$ with semiflow generated
by $G$. So the Dirichlet-to-Neumann operator is a multiplicative perturbation of 
the generator of  the semigroup of composition operators. Indeed, in \cite[Thm. 2.2]{EmSh13},
it has been shown that the Dirichlet-to-Neumann semigroup is the trace of a semigroup of composition operators only if $\Omega$ is a disk. This result relies
on the the fact that the normal unit vector (viewed as a complex valued
map on $\partial\Omega$) can only extended analytically to $\Omega$
if $\partial\Omega$ is a circle \cite[Thm 3.1]{EmSh13}.

\subsection*{Dirichlet-to-Robin semigroup}

From our previous investigations, it is now clear how to state well-posedness of the evolution
problem (\ref{eq:DTR}) associated with the Dirichlet-to-Robin operator. This is the main theorem of this article.
\begin{theorem}[Main Theorem]
Let $\Omega\subsetneq\C$ be a Jordan domain, and let $G:\Omega\to\C$
be the generator of a semiflow of holomorphic functions in $\cH(\Omega)$
and $g:\Omega\to\C$ holomorphic such that $X\subset\mathbb{H}(\Omega,\C)$
is $(g,G)$-admissible space. Then the evolution problem associated
with the Dirichlet-to-Robin operator
\begin{equation}
\begin{cases}
\partial_{t}u-g\cdot u-G\cdot\partial_{z}u=0 &\textrm{on } (0,\infty)\times\partial\Omega,\\
-\Delta u=0 & \textrm{on } (0,\infty)\times\partial\Omega,\\
u(0,\cdot)=u_{0} & \textrm{on }\partial\Omega,
\end{cases}
\end{equation}
is well-posed in $\partial X$, and the solution is given by the trace
of a weighted semigroup of composition operators.
\end{theorem}
\begin{proof}
Let $(S_{t})_{t}$ be the semigroup of weighted composition operators
with semiflow $(\varphi_{t})_{t}$ in $\cH(\Omega)$ generated by
$G$ and weight 
\[
m_{t}(z)=\exp\left(\int_{0}^{t}g(\varphi_{s}(z))ds\right)\quad(z\in\Omega).
\]
 We denote by $\Gamma$ the generator of $(S_{t})_{t}.$ Then the
Dirichlet-to-Robin operator $\dtr:\text{dom}(\dtr)\subset\partial X\to\partial X,u_{0}\mapsto(g\cdot u+G\cdot u')|_{\partial\Omega}$
is given by 
\begin{align*}
\dtr u_{0} &= \text{Tr}(g\cdot u+G\cdot u')\\
 &= \text{Tr}(\Gamma u).
\end{align*}
So we obtain the Dirichlet-to-Robin semigroup as 
\[
e^{-t\dtr}u_{0}=\text{Tr}(m_{t}\cdot u\circ\varphi_{t})\quad(u_{0}\in\partial X).
\]
\end{proof}
\begin{remark}
We would like to emphasize that a boundary space in the sense of distributions
is not necessary since we can always define boundary values using
hyperfunctions. In this case our initial value would be very general.
On the other hand, if $\varphi=\varphi_{1}$ has an interior Denjoy-Wolff
point and is not an inner function, then for $z\in\partial\Omega$
and $t$ sufficiently large, $\varphi_{t}(z)$ lies strictly inside
$\Omega$, see \cite[Thm. 1.2]{PC10}. Thus there is actually no
need to restrict to distributions in problem (\ref{eq:DTR}).

It is worth noting that the function $G$ may degenerate at some point $a\in\partial\Omega$.
This is even possible if $a$ is not a fixed point of the generated
semiflow $(\varphi_{t})_{t}$; on the other hand, if $a$ is a non-superrepulsive
fixed point of $\varphi$ (i.e., $\varphi'(a)\neq\infty$), then the
angular limit $\lim_{z\to a}G(z)=0$, see \cite[Thm. 1]{CMP06}. We repeat from the Introduction that we do not see how to include such a $G$ in the variational approach.
\end{remark}

\subsection*{Acknowledgment}
I am grateful to my supervisor Ralph Chill, who brought this topic to my attention, for support and valuable suggestions which improved the presentation of the paper.

\end{document}